\documentclass[a4paper,11pt]{article}

\usepackage{amsmath}
\usepackage{amssymb}
\usepackage{amsthm}
\usepackage{amsfonts}
\usepackage{mathtools}
\usepackage{xcolor}
\usepackage{url}

\theoremstyle{plain}
\newtheorem{theorem}{Theorem}[section]
\newtheorem{proposition}{Proposition}[section]
\newtheorem{lemma}{Lemma}[section]
\newtheorem{corollary}{Corollary}[section]

\theoremstyle{definition}
\newtheorem{definition}{Definition}[section]

\theoremstyle{remark}
\newtheorem{remark}{Remark}[section]

\newcommand{\Kf}{\operatorname{Kf}}
\newcommand{\HT}{\operatorname{HT}}

\begin{document}

\title{\textbf{A Decomposition of the Hitting Time Index Using the Kirchhoff Index and an Asymmetry Term}
\vspace{15mm}}

\author{
Shunya Tamura (corresponding author) \\
Okegawa City Okegawa West Junior High School \\
3680-1 Kawataya, Okegawa, Saitama 363-0027, Japan \\
Email: shunya.tamura059@gmail.com \\
\\ \\
Jos\'e Luis Palacios \\
Department of Electrical and Computer Engineering, University of New Mexico \\
Albuquerque, NM 87131, USA \\
jpalacios@unm.edu \\
\\ \\
Aleksandar Petojevi\'c \\
Faculty of Education, University of Novi Sad \\
Podgori\v{c}ka 4, 25000 Sombor, Republic of Serbia \\
apetoje@uns.ac.rs
}

\date{}

\thispagestyle{empty}
\maketitle
\vspace{50mm}
\clearpage
\begin{abstract}
We study the hitting time index,
a graph invariant defined in terms of expected hitting times
of simple random walks on finite connected graphs,
through its relation with the Kirchhoff index.
Using the commute time identity and Tetali's formula,
we decompose the hitting time index into a Kirchhoff-index term
and a nonnegative asymmetry term expressed by
degree-weighted effective resistance sums.
We derive lower and upper bounds for the asymmetry term,
together with their equality conditions,
and obtain a lower bound for the hitting time index
in terms of the range of the degree-weighted resistance sums.
For trees, we obtain lower and upper bounds for the asymmetry term,
with the upper bound expressed in terms of the Wiener index,
and characterize the equality case in the lower bound.
In particular, we show that among all trees of fixed order,
the hitting time index is minimized uniquely by the star graph.

Finally, we apply the decomposition to complete bipartite graphs,
paths, and graphs obtained by conjoining complete graphs,
and recover several known hitting-time formulas.
\end{abstract}


\section{Introduction}
\label{sec:introduction}
Random walks on graphs provide a natural connection
between probabilistic methods and structural graph theory.
Among the associated quantities, hitting times are closely related
to effective resistance and classical graph invariants
such as the Kirchhoff index.
Recently, Palacios \cite{Palacios2025}
introduced the hitting time index,
which aggregates the asymmetric behavior of pairwise hitting times.
From this point of view, relations among the hitting time index,
the Kirchhoff index, and the Wiener index provide a natural connection
between random walks on graphs and classical graph invariants.

Let $G=(V,E)$ be a finite connected simple graph,
and consider the simple random walk on $G$.
For vertices $u,v\in V$, let $H_{uv}$ denote
the expected hitting time from $u$ to $v$.
Hitting times are basic quantities in the study of random walks on graphs
and are closely related to effective resistance in electrical networks
\cite{DoyleSnell1984,ChandraEtAl1989,Tetali1991}.

The hitting time index is defined by
\[
\HT(G)
=
\sum_{\{u,v\}\subseteq V}
\max\{H_{uv},H_{vu}\}.
\]
Since then, several results on the hitting time index,
including bounds and explicit formulas for some graph families,
have been studied
\cite{PalaciosSantamaria2025,
HeEtAl2026,OrlicEtAl2026,PalaciosPetojevic2026}.

Regarding each edge as a unit resistor,
let $R_G(u,v)$ denote the effective resistance between $u$ and $v$.
The Kirchhoff index is defined by
\[
\Kf(G)
=
\sum_{\{u,v\}\subseteq V}
R_G(u,v)
\]
\cite{KleinRandic1993,GutmanMohar1996}.
If $G$ has $m$ edges, then the commute time identity gives
\[
H_{uv}+H_{vu}
=
2mR_G(u,v)
\]
\cite{ChandraEtAl1989,Tetali1991}.

Palacios \cite{Palacios2025} showed that
\[
\HT(G)\geq m\Kf(G).
\]
He also considered graphs satisfying
\[
H_{uv}=H_{vu}
\]
for every pair of vertices $u,v$,
calling them highly symmetric graphs.
For such graphs,
\[
\HT(G)=m\Kf(G).
\]
Palacios \cite{Palacios2025} also noted that
walk-regular graphs are highly symmetric.

In the present paper, we study the difference
\[
\HT(G)-m\Kf(G).
\]
Using the commute time identity together with Tetali's formula,
we show that this difference admits the explicit representation
\[
\frac12
\sum_{\{u,v\}\subseteq V}
|H_{uv}-H_{vu}|
=
\frac12
\sum_{\{u,v\}\subseteq V}
|R_d(u)-R_d(v)|,
\]
where
\[
R_d(u)
=
\sum_{x\in V}d(x)R_G(u,x).
\]
Accordingly, we define
\[
\mathcal A(G)
=
\frac12
\sum_{\{u,v\}\subseteq V}
|R_d(u)-R_d(v)|
\]
and write
\[
\HT(G)=m\Kf(G)+\mathcal A(G).
\]
We call $\mathcal A(G)$ the asymmetry term.
Thus, $\mathcal A(G)$ measures precisely the contribution to the hitting time index
arising from the asymmetry of hitting times.
The proof of this decomposition is given in
Section~\ref{sec:decomposition}.

We then study $\mathcal A(G)$ by arranging the values of $R_d$
in nondecreasing order and expressing it as a weighted sum
of consecutive differences.
From this representation, we obtain lower and upper bounds in terms of
\[
R_d^{\max}-R_d^{\min},
\]
together with their equality conditions.
Combining the lower bound with the decomposition gives
\[
\HT(G)
\geq
m\Kf(G)
+
\frac{n-1}{2}
\left(
R_d^{\max}-R_d^{\min}
\right).
\]

We also consider trees.
Since effective resistance is equal to graph distance in a tree,
the Kirchhoff index coincides with the Wiener index.
For every tree $T$ with $n\geq3$ vertices, we obtain
\[
(n-1)(n-2)
\leq
\mathcal A(T)
\leq
(n-2)W(T).
\]
The lower bound is attained if and only if
\[
T\cong K_{1,n-1}.
\]
Consequently,
\[
(n-1)W(T)+(n-1)(n-2)
\leq
\HT(T)
\leq
(2n-3)W(T).
\]
Combining the lower bound with the known minimum of the Wiener index among trees,
we also recover that the star minimizes the hitting time index among trees
of fixed order.

In Section~\ref{sec:examples},
we apply the decomposition to complete bipartite graphs,
path graphs, and graphs obtained by conjoining complete graphs.
For paths, the decomposition gives a short derivation
of the known formula for the hitting time index.
For conjoined complete graphs, the corresponding hitting time indices
can be recovered from the Kirchhoff index and the asymmetry term
without calculating all individual hitting times.

The paper is organized as follows.
In Section~\ref{sec:preliminaries},
we recall basic facts about hitting times,
effective resistance, the Kirchhoff index,
and the Wiener index.
In Section~\ref{sec:decomposition},
we discuss the decomposition of the hitting time index
and the asymmetry term.
In Section~\ref{sec:bounds},
we study a representation and bounds for the asymmetry term.
In Section~\ref{sec:trees},
we consider applications to trees.
In Section~\ref{sec:examples},
we give explicit calculations for several graph families.
Finally, in Section~\ref{sec:conclusion},
we summarize the results.

\section{Preliminaries}
\label{sec:preliminaries}
In this section, we recall some basic facts about hitting times,
effective resistance, the Kirchhoff index, and the Wiener index.

Throughout this paper, let $G=(V,E)$ be a finite connected simple graph
with $n\geq2$ vertices and $m$ edges.
For a vertex $u\in V$, let $d(u)$ denote its degree.

We consider the simple random walk on $G$.
For a vertex $v\in V$, let
\[
T_v=\inf\{t\geq0:X_t=v\}
\]
be the first hitting time of $v$, and define
\[
H_{uv}=\mathbb E_u[T_v].
\]
Thus, $H_{uv}$ is the expected hitting time from $u$ to $v$.

Throughout the paper, a sum over $\{u,v\}\subseteq V$
is taken over all unordered pairs of distinct vertices.
The hitting time index introduced by Palacios \cite{Palacios2025} is
\[
\HT(G)
=
\sum_{\{u,v\}\subseteq V}
\max\{H_{uv},H_{vu}\}.
\]

Regarding each edge as a unit resistor,
let $R_G(u,v)$ denote the effective resistance between $u$ and $v$.
The Kirchhoff index is defined by
\[
\Kf(G)
=
\sum_{\{u,v\}\subseteq V}
R_G(u,v)
\]
\cite{KleinRandic1993,GutmanMohar1996}.

For vertices $u,v\in V$, let $d_G(u,v)$ denote the graph distance.
The Wiener index is defined by
\[
W(G)
=
\sum_{\{u,v\}\subseteq V}
d_G(u,v).
\]

We first recall the commute time identity.

\begin{theorem}[\cite{ChandraEtAl1989,Tetali1991}]
\label{thm:commute}
For any $u,v\in V$,
\[
H_{uv}+H_{vu}
=
2mR_G(u,v).
\]
\end{theorem}

We also use Tetali's formula.

\begin{theorem}[\cite{Tetali1991}]
\label{thm:tetali}
For any $u,v\in V$,
\[
H_{uv}
=
mR_G(u,v)
+
\frac12
\sum_{x\in V}
d(x)
\left(
R_G(v,x)-R_G(u,x)
\right).
\]
\end{theorem}

For each vertex $u\in V$, define the degree-weighted resistance sum by
\[
R_d(u)
=
\sum_{x\in V}
d(x)R_G(u,x).
\]
With this notation, Tetali's formula becomes
\[
H_{uv}
=
mR_G(u,v)
+
\frac12
\left(
R_d(v)-R_d(u)
\right).
\]
Interchanging $u$ and $v$ gives
\[
H_{vu}
=
mR_G(u,v)
+
\frac12
\left(
R_d(u)-R_d(v)
\right).
\]
Subtracting the two expressions, we obtain
\[
H_{uv}-H_{vu}
=
R_d(v)-R_d(u).
\]
This identity will be used throughout the paper.

Palacios \cite{Palacios2025} called a graph
\emph{highly symmetric} if
\[
H_{uv}=H_{vu}
\]
for every pair of vertices $u,v\in V$.
Graphs with symmetric hitting times were further studied by
Chang and Xu \cite{ChangXu2017}.

The preceding identity shows that
\[
H_{uv}=H_{vu}
\quad\Longleftrightarrow\quad
R_d(u)=R_d(v).
\]
Therefore, $G$ is highly symmetric if and only if
$R_d(u)$ is independent of $u\in V$.

Palacios \cite{Palacios2025} showed that
\[
\HT(G)\geq m\Kf(G),
\]
with equality for highly symmetric graphs.
He also noted that walk-regular graphs are highly symmetric.

\section{Decomposition of the hitting time index}
\label{sec:decomposition}

In this section, we decompose the hitting time index
into a Kirchhoff-index term and a term measuring
the asymmetry of hitting times.

For any real numbers $a$ and $b$,
\[
\max\{a,b\}
=
\frac{a+b+|a-b|}{2}.
\]
Applying this identity to each unordered pair of vertices
gives the following decomposition.

\begin{theorem}
\label{thm:HT-decomposition}

Let $G=(V,E)$ be a finite connected simple graph
with $n\geq2$ vertices and $m$ edges. Then
\[
\HT(G)
=
m\Kf(G)
+
\frac12
\sum_{\{u,v\}\subseteq V}
|H_{uv}-H_{vu}|.
\]
Using the degree-weighted resistance sums defined in
Section~\ref{sec:preliminaries}, this can also be written as
\[
\HT(G)
=
m\Kf(G)
+
\frac12
\sum_{\{u,v\}\subseteq V}
|R_d(u)-R_d(v)|.
\]

\end{theorem}

\begin{proof}

For distinct vertices $u,v\in V$,
\[
\max\{H_{uv},H_{vu}\}
=
\frac{
H_{uv}+H_{vu}
+
|H_{uv}-H_{vu}|
}{2}.
\]
By Theorem~\ref{thm:commute},
\[
H_{uv}+H_{vu}
=
2mR_G(u,v),
\]
and hence
\[
\max\{H_{uv},H_{vu}\}
=
mR_G(u,v)
+
\frac12
|H_{uv}-H_{vu}|.
\]
Summing over all unordered pairs $\{u,v\}\subseteq V$ gives
\[
\begin{aligned}
\HT(G)
&=
m
\sum_{\{u,v\}\subseteq V}
R_G(u,v)
+
\frac12
\sum_{\{u,v\}\subseteq V}
|H_{uv}-H_{vu}|\\
&=
m\Kf(G)
+
\frac12
\sum_{\{u,v\}\subseteq V}
|H_{uv}-H_{vu}|.
\end{aligned}
\]
From Section~\ref{sec:preliminaries},
\[
H_{uv}-H_{vu}
=
R_d(v)-R_d(u),
\]
so
\[
|H_{uv}-H_{vu}|
=
|R_d(u)-R_d(v)|.
\]
This gives the second form of the decomposition.

\end{proof}

\begin{definition}
\label{def:asymmetry}

For a finite connected simple graph $G$, define
\[
\mathcal A(G)
=
\frac12
\sum_{\{u,v\}\subseteq V}
|H_{uv}-H_{vu}|.
\]
We call $\mathcal A(G)$ the asymmetry term of $G$.
Equivalently, by Theorem~\ref{thm:HT-decomposition},
\[
\mathcal A(G)
=
\frac12
\sum_{\{u,v\}\subseteq V}
|R_d(u)-R_d(v)|.
\]

\end{definition}


With this notation, the decomposition takes the form
\[
\HT(G)=m\Kf(G)+\mathcal A(G).
\]
Thus, $\mathcal A(G)$ measures the contribution to the hitting time index
coming from the asymmetry between $H_{uv}$ and $H_{vu}$.

\begin{remark}
\label{rem:palacios-bound}

Since $\mathcal A(G)\geq0$, the decomposition recovers the inequality
\[
\HT(G)\geq m\Kf(G)
\]
proved by Palacios \cite{Palacios2025}.

Moreover,
\[
\mathcal A(G)=0
\quad\Longleftrightarrow\quad
H_{uv}=H_{vu}
\qquad (u,v\in V).
\]
Equivalently, by the identity from Section~\ref{sec:preliminaries},
this occurs if and only if $R_d(u)$ is independent of $u\in V$.

Therefore, $\mathcal A(G)=0$ precisely for the highly symmetric graphs
considered by Palacios. In this case, the decomposition reduces to
\[
\HT(G)=m\Kf(G),
\]
in agreement with \cite{Palacios2025}.

\end{remark}


\section{Bounds for the asymmetry term}
\label{sec:bounds}
In this section, we consider bounds for $\mathcal A(G)$ in terms of
the range of the degree-weighted resistance sums. Put
\[
R_d^{\min}=\min_{u\in V}R_d(u),
\qquad
R_d^{\max}=\max_{u\in V}R_d(u).
\]
We first express $\mathcal A(G)$ in terms of consecutive differences
between the ordered values of $R_d$.
\begin{proposition}
\label{prop:asymmetry-gaps}
Arrange the vertices so that
\[
R_d(v_1)\leq R_d(v_2)\leq\cdots\leq R_d(v_n),
\]
and put
\[
\Delta_k
=
R_d(v_{k+1})-R_d(v_k)
\qquad
(1\leq k\leq n-1).
\]
Then
\[
\mathcal A(G)
=
\frac12
\sum_{k=1}^{n-1}
k(n-k)\Delta_k.
\]
\end{proposition}
\begin{proof}
By Definition~\ref{def:asymmetry},
\[
2\mathcal A(G)
=
\sum_{1\leq i<j\leq n}
|R_d(v_i)-R_d(v_j)|.
\]
Since the vertices are arranged in nondecreasing order,
\[
|R_d(v_i)-R_d(v_j)|
=
R_d(v_j)-R_d(v_i)
=
\sum_{k=i}^{j-1}\Delta_k
\]
for $i<j$. Therefore,
\[
2\mathcal A(G)
=
\sum_{1\leq i<j\leq n}
\sum_{k=i}^{j-1}\Delta_k.
\]
For a fixed $k$, the term $\Delta_k$ appears precisely for those pairs
$(i,j)$ satisfying
\[
i\leq k<j.
\]
There are $k(n-k)$ such pairs. Hence,
\[
2\mathcal A(G)
=
\sum_{k=1}^{n-1}k(n-k)\Delta_k,
\]
which proves the formula.
\end{proof}
Since
\[
\sum_{k=1}^{n-1}\Delta_k
=
R_d^{\max}-R_d^{\min},
\]
Proposition~\ref{prop:asymmetry-gaps} gives the following bounds.
\begin{theorem}
\label{thm:asymmetry-range}
Let $G=(V,E)$ be a finite connected simple graph with $n\geq2$ vertices.
Then
\[
\frac{n-1}{2}
\left(
R_d^{\max}-R_d^{\min}
\right)
\leq
\mathcal A(G)
\leq
\frac12
\left\lfloor\frac{n^2}{4}\right\rfloor
\left(
R_d^{\max}-R_d^{\min}
\right).
\]
For the equality conditions, arrange the vertices so that
\[
R_d(v_1)\leq R_d(v_2)\leq\cdots\leq R_d(v_n).
\]
For $n\geq4$, equality in the lower bound holds if and only if
\[
R_d(v_2)=R_d(v_3)=\cdots=R_d(v_{n-1}).
\]
For $n=2$ and $n=3$, equality in the lower bound always holds.

For the upper bound, the equality conditions are as follows.
\begin{enumerate}
\item If $n=2r$, equality holds if and only if
\[
R_d(v_1)=\cdots=R_d(v_r)
\qquad
\text{and}
\qquad
R_d(v_{r+1})=\cdots=R_d(v_{2r}).
\]
\item If $n=2r+1$, equality holds if and only if
\[
R_d(v_1)=\cdots=R_d(v_r)
\qquad
\text{and}
\qquad
R_d(v_{r+2})=\cdots=R_d(v_{2r+1}).
\]
\end{enumerate}
\end{theorem}
\begin{proof}
By Proposition~\ref{prop:asymmetry-gaps},
\[
2\mathcal A(G)
=
\sum_{k=1}^{n-1}k(n-k)\Delta_k.
\]
For $1\leq k\leq n-1$,
\[
k(n-k)-(n-1)
=
(k-1)(n-k-1)
\geq0.
\]
Hence,
\[
k(n-k)\geq n-1,
\]
and therefore
\[
\begin{aligned}
2\mathcal A(G)
&\geq
(n-1)\sum_{k=1}^{n-1}\Delta_k\\
&=
(n-1)
\left(
R_d^{\max}-R_d^{\min}
\right).
\end{aligned}
\]
This proves the lower bound.

Suppose first that $n\geq4$. Since
\[
k(n-k)>n-1
\]
precisely for
\[
2\leq k\leq n-2,
\]
equality in the lower bound holds if and only if
\[
\Delta_k=0
\qquad
(2\leq k\leq n-2).
\]
Equivalently,
\[
R_d(v_2)=\cdots=R_d(v_{n-1}).
\]
If $n=2$ or $n=3$, all coefficients satisfy
\[
k(n-k)=n-1,
\]
so equality in the lower bound always holds.

For the upper bound, we have
\[
k(n-k)
\leq
\left\lfloor\frac{n^2}{4}\right\rfloor
\]
for every $1\leq k\leq n-1$. Thus,
\[
2\mathcal A(G)
\leq
\left\lfloor\frac{n^2}{4}\right\rfloor
\sum_{k=1}^{n-1}\Delta_k,
\]
and therefore
\[
\mathcal A(G)
\leq
\frac12
\left\lfloor\frac{n^2}{4}\right\rfloor
\left(
R_d^{\max}-R_d^{\min}
\right).
\]

If $n=2r$, the maximum of $k(n-k)$ is attained only at $k=r$.
Hence equality holds if and only if
\[
\Delta_k=0
\qquad
(k\neq r),
\]
which is equivalent to
\[
R_d(v_1)=\cdots=R_d(v_r),
\qquad
R_d(v_{r+1})=\cdots=R_d(v_{2r}).
\]

If $n=2r+1$, the maximum is attained precisely at
$k=r$ and $k=r+1$.
Hence equality holds if and only if
\[
\Delta_k=0
\qquad
(k\neq r,r+1),
\]
which is equivalent to
\[
R_d(v_1)=\cdots=R_d(v_r),
\qquad
R_d(v_{r+2})=\cdots=R_d(v_{2r+1}).
\]
\end{proof}
Combining the lower bound in Theorem~\ref{thm:asymmetry-range}
with Theorem~\ref{thm:HT-decomposition} gives the following consequence.
\begin{corollary}
\label{cor:HT-range-bound}
Let $G=(V,E)$ be a finite connected simple graph
with $n\geq2$ vertices and $m$ edges. Then
\[
\HT(G)
\geq
m\Kf(G)
+
\frac{n-1}{2}
\left(
R_d^{\max}-R_d^{\min}
\right).
\]
For the equality conditions, arrange the vertices so that
\[
R_d(v_1)\leq R_d(v_2)\leq\cdots\leq R_d(v_n).
\]
For $n\geq4$, equality holds if and only if
\[
R_d(v_2)=\cdots=R_d(v_{n-1}).
\]
For $n=2$ and $n=3$, equality always holds.
\end{corollary}
\begin{proof}
By Theorem~\ref{thm:HT-decomposition},
\[
\HT(G)
=
m\Kf(G)+\mathcal A(G).
\]
The result follows directly from the lower bound
and its equality condition in Theorem~\ref{thm:asymmetry-range}.
\end{proof}

\section{Applications to trees}
\label{sec:trees}
In this section, we apply the decomposition to trees.
Let $T$ be a tree with $n\geq2$ vertices.
Since
\[
|E(T)|=n-1
\]
and the effective resistance between two vertices of a tree
is equal to their graph distance, we have
\[
\Kf(T)=W(T).
\]
Thus, Theorem~\ref{thm:HT-decomposition} gives
\[
\HT(T)=(n-1)W(T)+\mathcal A(T).
\]
We first record a relation between the degree-weighted resistance sums
at adjacent vertices.

\begin{lemma}
\label{lem:tree-edge-Rd}
Let $u$ and $v$ be adjacent vertices of a tree $T$ with $n$ vertices.
Removing the edge $uv$ separates $T$ into two components
$T_u$ and $T_v$ containing $u$ and $v$, respectively.
If
\[
|V(T_u)|=s,
\qquad
|V(T_v)|=n-s,
\]
then
\[
R_d(u)-R_d(v)=2(n-2s).
\]
\end{lemma}

\begin{proof}
Since effective resistance is equal to graph distance in a tree,
\[
R_d(u)-R_d(v)
=
\sum_{x\in V(T)}
d(x)
\left(
d_T(u,x)-d_T(v,x)
\right).
\]
For $x\in V(T_u)$,
\[
d_T(u,x)-d_T(v,x)=-1,
\]
whereas for $x\in V(T_v)$,
\[
d_T(u,x)-d_T(v,x)=1.
\]
Hence
\[
R_d(u)-R_d(v)
=
-\sum_{x\in V(T_u)}d(x)
+
\sum_{x\in V(T_v)}d(x).
\]
Here the degrees are taken in the original tree $T$.
Since exactly one edge of $T$ joins $T_u$ and $T_v$,
\[
\sum_{x\in V(T_u)}d(x)
=
2|E(T_u)|+1
=
2(s-1)+1
=
2s-1,
\]
and similarly,
\[
\sum_{x\in V(T_v)}d(x)
=
2(n-s-1)+1
=
2(n-s)-1.
\]
Therefore,
\[
\begin{aligned}
R_d(u)-R_d(v)
&=
-(2s-1)+\bigl(2(n-s)-1\bigr)\\
&=
2(n-2s).
\end{aligned}
\]
\end{proof}

In particular, if $u$ is a leaf and $v$ is its unique neighbor,
then $s=1$, and hence
\[
R_d(u)-R_d(v)=2(n-2).
\]
Combining this observation with the lower bound in
Theorem~\ref{thm:asymmetry-range} gives the following result.

\begin{theorem}
\label{thm:tree-lower-bound}
Let $T$ be a tree with $n\geq3$ vertices. Then
\[
\mathcal A(T)\geq(n-1)(n-2)
\]
and
\[
\HT(T)
\geq
(n-1)W(T)+(n-1)(n-2).
\]
Equality in either inequality holds if and only if
\[
T\cong K_{1,n-1}.
\]
\end{theorem}

\begin{proof}
Let $u$ be a leaf of $T$ and let $v$ be its unique neighbor.
By Lemma~\ref{lem:tree-edge-Rd},
\[
R_d(u)-R_d(v)=2(n-2).
\]
Therefore,
\[
R_d^{\max}-R_d^{\min}
\geq
2(n-2).
\]
By Theorem~\ref{thm:asymmetry-range},
\[
\begin{aligned}
\mathcal A(T)
&\geq
\frac{n-1}{2}
\left(
R_d^{\max}-R_d^{\min}
\right)\\
&\geq
\frac{n-1}{2}\,2(n-2)\\
&=
(n-1)(n-2).
\end{aligned}
\]
Since
\[
\HT(T)=(n-1)W(T)+\mathcal A(T),
\]
we also obtain
\[
\HT(T)
\geq
(n-1)W(T)+(n-1)(n-2).
\]

We next consider the equality case.
First suppose that
\[
T\cong K_{1,n-1}.
\]
Let $c$ be the center and let $l$ be any leaf.
By Lemma~\ref{lem:tree-edge-Rd},
\[
R_d(l)-R_d(c)=2(n-2).
\]
By symmetry, all leaves have the same value of $R_d$.
Thus, only center--leaf pairs contribute to $\mathcal A(T)$, and hence
\[
\mathcal A(T)
=
\frac12(n-1)\,2(n-2)
=
(n-1)(n-2).
\]
Therefore, equality holds.

Conversely, suppose that
\[
\mathcal A(T)=(n-1)(n-2).
\]
Since
\[
\mathcal A(T)
\geq
\frac{n-1}{2}
\left(
R_d^{\max}-R_d^{\min}
\right)
\geq
(n-1)(n-2),
\]
equality must hold in both inequalities. Hence
\[
R_d^{\max}-R_d^{\min}
=
2(n-2).
\]
If $n=3$, then
\[
T\cong P_3=K_{1,2},
\]
so the statement follows.

Now suppose that $n\geq4$ and arrange the vertices so that
\[
R_d(v_1)\leq R_d(v_2)\leq\cdots\leq R_d(v_n).
\]
By the equality condition in Theorem~\ref{thm:asymmetry-range},
\[
R_d(v_2)=\cdots=R_d(v_{n-1}).
\]
Put
\[
\rho
=
R_d(v_2)
=
\cdots
=
R_d(v_{n-1}).
\]
Thus, at most one vertex has $R_d$-value smaller than $\rho$,
and at most one vertex has $R_d$-value larger than $\rho$.

Let $u$ be any leaf and let $v$ be its unique neighbor.
By Lemma~\ref{lem:tree-edge-Rd},
\[
R_d(u)-R_d(v)
=
2(n-2)
=
R_d^{\max}-R_d^{\min}.
\]
Since the difference \(R_d(u)-R_d(v)\) is equal to the full range
\(R_d^{\max}-R_d^{\min}\), it follows that
\[
R_d(u)=R_d^{\max},
\qquad
R_d(v)=R_d^{\min}.
\]
Every tree has at least two leaves, and, as shown above, every leaf has
\(R_d\)-value \(R_d^{\max}\). Hence at least two vertices attain
\(R_d^{\max}\). Since at most one vertex can have an \(R_d\)-value
strictly larger than \(\rho\), it follows that
\[
R_d^{\max}=\rho.
\]
Consequently, $R_d^{\min}<\rho$, and there is exactly one vertex $c$ satisfying
\[
R_d(c)=R_d^{\min}.
\]
Since the neighbor of every leaf has degree-weighted resistance sum
$R_d^{\min}$, every leaf is adjacent to $c$.

If some vertex had distance at least two from $c$,
choose a vertex $w$ at maximum distance from $c$.
Since $T$ is a tree, $w$ must be a leaf.
Moreover,
\[
d_T(c,w)\geq2,
\]
so $w$ is not adjacent to $c$,
contradicting the fact that every leaf is adjacent to $c$.
Hence every vertex other than $c$ is adjacent to $c$, and therefore
\[
T\cong K_{1,n-1}.
\]

Finally, since
\[
\HT(T)=(n-1)W(T)+\mathcal A(T),
\]
equality in
\[
\HT(T)
\geq
(n-1)W(T)+(n-1)(n-2)
\]
is equivalent to
\[
\mathcal A(T)=(n-1)(n-2).
\]
Therefore, equality in either inequality holds if and only if
\[
T\cong K_{1,n-1}.
\]
\end{proof}

We next give an upper bound for the asymmetry term of a tree.

\begin{theorem}
\label{thm:tree-upper-bound}

Let $T$ be a tree with $n\geq3$ vertices. Then
\[
\mathcal A(T)\leq(n-2)W(T).
\]
Consequently,
\[
\HT(T)\leq(2n-3)W(T).
\]

\end{theorem}

\begin{proof}
Let $u,v\in V(T)$, and let
\[
u=v_0,v_1,\ldots,v_k=v
\]
be the unique path from $u$ to $v$, where
\[
k=d_T(u,v).
\]
For each $1\leq i\leq k$, removing the edge
$v_{i-1}v_i$ separates $T$ into two components.
Let $s_i$ be the number of vertices in the component
containing $v_{i-1}$.

By Lemma~\ref{lem:tree-edge-Rd},
\[
R_d(v_{i-1})-R_d(v_i)
=
2(n-2s_i).
\]
Therefore,
\[
R_d(u)-R_d(v)
=
\sum_{i=1}^{k}2(n-2s_i).
\]
Since $1\leq s_i\leq n-1$, we have
\[
|n-2s_i|\leq n-2.
\]
Hence,
\[
\begin{aligned}
|R_d(u)-R_d(v)|
&\leq
\sum_{i=1}^{k}2|n-2s_i|\\
&\leq
2(n-2)k\\
&=
2(n-2)d_T(u,v).
\end{aligned}
\]
Summing over all unordered pairs of vertices gives
\[
\begin{aligned}
\mathcal A(T)
&=
\frac12
\sum_{\{u,v\}\subseteq V(T)}
|R_d(u)-R_d(v)|\\
&\leq
(n-2)
\sum_{\{u,v\}\subseteq V(T)}
d_T(u,v)\\
&=
(n-2)W(T).
\end{aligned}
\]
Finally, since
\[
\HT(T)=(n-1)W(T)+\mathcal A(T),
\]
we obtain
\[
\HT(T)
\leq
(n-1)W(T)+(n-2)W(T)
=
(2n-3)W(T).
\]

\end{proof}

\begin{remark}
\label{rem:tree-upper-comparison}

The general inequality
\[
\HT(G)\leq2m\Kf(G)
\]
was obtained by Palacios \cite{Palacios2025}.
Together with the decomposition
\[
\HT(G)=m\Kf(G)+\mathcal A(G),
\]
it gives
\[
\mathcal A(G)\leq m\Kf(G).
\]
For a tree $T$, this yields
\[
\mathcal A(T)\leq(n-1)W(T).
\]
Thus, Theorem~\ref{thm:tree-upper-bound} gives the sharper estimate
\[
\mathcal A(T)\leq(n-2)W(T).
\]
\end{remark}

\begin{remark}
\label{rem:star-minimizer}
It is well known that, among all trees on \(n\) vertices,
\[
W(T)\geq(n-1)^2,
\]
with equality if and only if
\[
T\cong S_n=K_{1,n-1}
\]
\cite{DobryninEntringerGutman2001}.
Therefore, by Theorem~\ref{thm:tree-lower-bound},
every tree \(T\) on \(n\geq3\) vertices satisfies
\[
\begin{aligned}
\HT(T)
&\geq
(n-1)W(T)+(n-1)(n-2)\\
&\geq
(n-1)(n-1)^2+(n-1)(n-2)\\
&=
(n-1)(n^2-n-1).
\end{aligned}
\]
On the other hand,
\[
\HT(S_n)=(n-1)(n^2-n-1),
\]
as given in \cite{Palacios2025}.
Hence, among all trees on $n$ vertices,
the hitting time index is minimized uniquely by the star $S_n$.

It is worth noting that the star also occurs as an extremal tree
for a number of vertex-degree-based topological indices;
see, for example, Cruz and Rada \cite{Rada2019}.
\end{remark}

\section{Examples}
\label{sec:examples}
In this section, we apply the decomposition to several graph families.
We first consider complete bipartite graphs and path graphs,
and then graphs obtained by conjoining complete graphs.

\subsection{Complete bipartite graphs}
Let $p,q\geq1$, and let $K_{p,q}$ have partite sets $X$ and $Y$ with
\[
|X|=p,
\qquad
|Y|=q.
\]
Then
\[
d(x)=q
\quad (x\in X),
\qquad
\text{and}
\qquad
d(y)=p
\qquad (y\in Y).
\]
The effective resistances are
\[
R_G(x,x')=\frac{2}{q}
\quad
(x,x'\in X,\ x\neq x'),
\qquad
R_G(y,y')=\frac{2}{p}
\quad
(y,y'\in Y,\ y\neq y'),
\]
and
\[
R_G(x,y)
=
\frac{p+q-1}{pq}
\qquad
(x\in X,\ y\in Y).
\]
These formulas give the following values of the degree-weighted resistance sums.

\begin{proposition}
\label{prop:complete-bipartite}

For $G=K_{p,q}$,
\[
R_d(x)=3p+q-3
\quad
(x\in X)
\qquad
\text{and}
\qquad
R_d(y)=p+3q-3
\quad
(y\in Y).
\]
Consequently,
\[
\mathcal A(K_{p,q})=pq|p-q|.
\]

\end{proposition}

\begin{proof}

For $x\in X$,
\[
\begin{aligned}
R_d(x)
&=
\sum_{\substack{x'\in X\\x'\neq x}}
d(x')R_G(x,x')
+
\sum_{y\in Y}
d(y)R_G(x,y)\\
&=
(p-1)q\frac{2}{q}
+
qp\frac{p+q-1}{pq}\\
&=
3p+q-3.
\end{aligned}
\]
Similarly,
\[
\begin{aligned}
R_d(y)
&=
\sum_{x\in X}
d(x)R_G(y,x)
+
\sum_{\substack{y'\in Y\\y'\neq y}}
d(y')R_G(y,y')\\
&=
pq\frac{p+q-1}{pq}
+
(q-1)p\frac{2}{p}\\
&=
p+3q-3.
\end{aligned}
\]
Thus, vertices in the same partite set have the same value of $R_d$,
while for $x\in X$ and $y\in Y$,
\[
|R_d(x)-R_d(y)|
=
2|p-q|.
\]
There are $pq$ such pairs, and hence
\[
\mathcal A(K_{p,q})
=
\frac12 pq\cdot2|p-q|
=
pq|p-q|.
\]

\end{proof}


In particular,
\[
\mathcal A(K_{p,q})=0
\quad\Longleftrightarrow\quad
p=q.
\]
The Kirchhoff index is
\[
\Kf(K_{p,q})
=
\frac{p(p-1)}{q}
+
\frac{q(q-1)}{p}
+
p+q-1.
\]
Since $|E(K_{p,q})|=pq$, the decomposition gives
\[
\begin{aligned}
\HT(K_{p,q})
&=
pq\,\Kf(K_{p,q})
+
pq|p-q|\\
&=
p^2(p-1)
+
q^2(q-1)
+
pq(p+q-1)
+
pq|p-q|.
\end{aligned}
\]
This agrees with the formula for $K_{p,q}$ given in
\cite{PalaciosSantamariaGraphs2025}.

\subsection{Path graphs}

We next consider the path graph $P_n$ with vertices
$1,2,\ldots,n$ in their natural order.
Since $P_n$ is a tree,
\[
R_{P_n}(i,j)=|i-j|.
\]
Also,
\[
d(1)=d(n)=1,
\qquad
d(j)=2
\quad
(2\leq j\leq n-1).
\]

\begin{lemma}
\label{lem:path-Rd}

For $1\leq i\leq n$,
\[
R_d(i)
=
2i^2-2(n+1)i+n^2+1.
\]
In particular,
\[
R_d(i)=R_d(n+1-i).
\]

\end{lemma}

\begin{proof}

By definition,
\[
R_d(i)
=
\sum_{j=1}^{n}d(j)|i-j|.
\]
Counting all degrees first as $2$ and then subtracting one contribution
from each end vertex gives
\[
R_d(i)
=
2\sum_{j=1}^{n}|i-j|
-(i-1)-(n-i).
\]
Since
\[
\sum_{j=1}^{n}|i-j|
=
\frac{i(i-1)}{2}
+
\frac{(n-i)(n-i+1)}{2},
\]
we obtain
\[
R_d(i)
=
2i^2-2(n+1)i+n^2+1.
\]
The symmetry
\[
R_d(i)=R_d(n+1-i)
\]
follows immediately.

\end{proof}

\begin{theorem}
\label{thm:path-asymmetry}

For $n\geq2$,
\[
\mathcal A(P_n)
=
\begin{cases}
\dfrac{n^2(n^2-4)}{24},
& n \text{ is even},\\[8pt]
\dfrac{(n^2-1)(n^2-3)}{24},
& n \text{ is odd}.
\end{cases}
\]

\end{theorem}

\begin{proof}

First, let $n=2r$.
By Lemma~\ref{lem:path-Rd},
\[
R_d(r-s)
=
2r^2-2r+1+2s(s+1)
\qquad
(0\leq s\leq r-1),
\]
and each of these values occurs twice. Therefore,
\[
\mathcal A(P_{2r})
=
4\sum_{0\leq s<t\leq r-1}
\left(
t(t+1)-s(s+1)
\right).
\]
Using
\[
\sum_{s=0}^{t-1}
\left(
t(t+1)-s(s+1)
\right)
=
\frac{t(t+1)(2t+1)}{3},
\]
we obtain
\[
\mathcal A(P_{2r})
=
\frac{2r^2(r^2-1)}{3}.
\]
Substituting $n=2r$ gives
\[
\mathcal A(P_n)
=
\frac{n^2(n^2-4)}{24}.
\]

Now let $n=2r+1$. Then
\[
R_d(r+1-s)
=
2(r^2+s^2)
\qquad
(0\leq s\leq r).
\]
The value corresponding to $s=0$ occurs once,
while each value corresponding to $1\leq s\leq r$ occurs twice.
Hence,
\[
\mathcal A(P_{2r+1})
=
2\sum_{s=1}^{r}s^2
+
4\sum_{1\leq s<t\leq r}(t^2-s^2).
\]
For the second sum,
\[
\begin{aligned}
\sum_{1\leq s<t\leq r}(t^2-s^2)
&=
\sum_{t=2}^{r}
\left(
(t-1)t^2-\sum_{s=1}^{t-1}s^2
\right)\\
&=
\frac{r(r-1)(r+1)^2}{6}.
\end{aligned}
\]
Using
\[
\sum_{s=1}^{r}s^2
=
\frac{r(r+1)(2r+1)}{6},
\]
we obtain
\[
\mathcal A(P_{2r+1})
=
\frac{r(r+1)(2r^2+2r-1)}{3}.
\]
Since $n=2r+1$,
\[
r(r+1)=\frac{n^2-1}{4}
\]
and
\[
2r^2+2r-1
=
\frac{n^2-3}{2}.
\]
Therefore,
\[
\mathcal A(P_n)
=
\frac{(n^2-1)(n^2-3)}{24}.
\]

\end{proof}


Since
\[
\Kf(P_n)=W(P_n)=\frac{n(n^2-1)}{6}
\]
and $|E(P_n)|=n-1$, we obtain the following formula.

\begin{corollary}
\label{cor:path-HT}

For $n\geq2$,
\[
\HT(P_n)
=
\begin{cases}
\dfrac{n(n-1)(n^2-1)}{6}
+
\dfrac{n^2(n^2-4)}{24},
& n \text{ is even},\\[10pt]
\dfrac{n(n-1)(n^2-1)}{6}
+
\dfrac{(n^2-1)(n^2-3)}{24},
& n \text{ is odd}.
\end{cases}
\]

\end{corollary}

\begin{proof}

This follows directly from
\[
\HT(P_n)
=
(n-1)\Kf(P_n)+\mathcal A(P_n)
\]
and Theorem~\ref{thm:path-asymmetry}.

\end{proof}

The formula in Corollary~\ref{cor:path-HT}
recovers the path formula of Palacios \cite{Palacios2025}.
For $n\geq3$, it also agrees with the case $k=1$
of the formula in \cite{PetojevicPalaciosOrlic2026}.

\begin{remark}
\label{rem:path-asymptotic-bound}
The formula in Theorem~\ref{thm:path-asymmetry} gives
\[
\mathcal A(P_n)\sim\frac{n^4}{24}
\qquad
(n\to\infty).
\]
On the other hand, applying the upper bound in
Theorem~\ref{thm:asymmetry-range} to $P_n$ gives an upper bound
asymptotic to
\[
\frac{n^4}{16}.
\]
Consequently, the ratio of $\mathcal A(P_n)$ to this upper bound
tends to
\[
\frac23.
\]
Thus, the general upper bound in
Theorem~\ref{thm:asymmetry-range} is not asymptotically attained
by paths.

\end{remark}

\subsection{Conjoined complete graphs}

We finally consider graphs obtained by identifying one vertex
of two or more complete graphs.
In these graphs, the effective resistances take only a few values,
so the asymmetry term can be computed directly.

Let $r,t\geq2$, and let $G_{r,t}$ be the graph obtained from
$K_r$ and $K_t$ by identifying one vertex from each graph.
Denote the common vertex by $c$, and put
\[
X=V(K_r)\setminus\{c\},
\qquad
Y=V(K_t)\setminus\{c\}.
\]
Thus,
\[
|X|=r-1,
\qquad
|Y|=t-1.
\]

Since the two complete graphs meet only at the cut vertex $c$,
attaching one block does not change the effective resistance
between vertices of the other block.
Moreover, for vertices belonging to different blocks,
the effective resistances add in series through $c$.
Therefore, for $x\in X$ and $y\in Y$, the effective resistances are
\[
R_G(c,x)=\frac{2}{r},
\qquad
R_G(c,y)=\frac{2}{t},
\]
\[
R_G(x,x')=\frac{2}{r}
\quad
(x,x'\in X,\ x\neq x'),
\qquad
R_G(y,y')=\frac{2}{t}
\quad
(y,y'\in Y,\ y\neq y'),
\]
and
\[
R_G(x,y)
=
\frac{2}{r}+\frac{2}{t}.
\]

\begin{proposition}
\label{prop:conjoined-two-complete}

For $G_{r,t}$,
\[
R_d(x)-R_d(c)
=
\frac{2t(t-1)}{r}
\qquad
(x\in X)
\]
and
\[
R_d(y)-R_d(c)
=
\frac{2r(r-1)}{t}
\qquad
(y\in Y).
\]
Consequently,
\[
\mathcal A(G_{r,t})
=
(r-1)(t-1)
\left\{
\frac{t}{r}
+
\frac{r}{t}
+
\frac{|t-r|}{rt}
\left(
t^2+rt+r^2-t-r
\right)
\right\}.
\]

\end{proposition}

\begin{proof}

The common vertex has degree
\[
d(c)=r+t-2,
\]
while
\[
d(x)=r-1
\quad (x\in X)
\qquad
\text{and}
\qquad
d(y)=t-1
\qquad (y\in Y).
\]
Using the resistance formulas above, we first have
\[
R_d(c)
=
\frac{2(r-1)^2}{r}
+
\frac{2(t-1)^2}{t}.
\]
For $x\in X$,
\[
\begin{aligned}
R_d(x)
&=
(r+t-2)\frac{2}{r}
+
(r-2)(r-1)\frac{2}{r}
+
(t-1)^2
\left(
\frac{2}{r}+\frac{2}{t}
\right).
\end{aligned}
\]
Subtracting $R_d(c)$ gives
\[
R_d(x)-R_d(c)
=
\frac{2t(t-1)}{r}.
\]
Similarly, for $y\in Y$,
\[
R_d(y)-R_d(c)
=
\frac{2r(r-1)}{t}.
\]
Put
\[
a=\frac{2t(t-1)}{r},
\qquad
b=\frac{2r(r-1)}{t}.
\]
All vertices in $X$ have the value $R_d(c)+a$,
and all vertices in $Y$ have the value $R_d(c)+b$.
Therefore,
\[
\mathcal A(G_{r,t})
=
\frac12
\left\{
(r-1)a
+
(t-1)b
+
(r-1)(t-1)|a-b|
\right\}.
\]
Moreover,
\[
|a-b|
=
\frac{2|t-r|}{rt}
\left(
t^2+rt+r^2-t-r
\right).
\]
Substitution gives the stated formula.

\end{proof}

The pairs contained in the $K_r$ block contribute
\[
\binom{r}{2}\frac{2}{r}
=
r-1,
\]
and those contained in the $K_t$ block contribute
\[
\binom{t}{2}\frac{2}{t}
=
t-1.
\]
The $(r-1)(t-1)$ pairs with one vertex in $X$ and one in $Y$
have resistance
\[
\frac{2}{r}+\frac{2}{t}.
\]
Hence,
\[
\begin{aligned}
\Kf(G_{r,t})
&=
(r-1)+(t-1)+(r-1)(t-1)
\left(
\frac{2}{r}+\frac{2}{t}
\right)\\
&=
r+t-2
+
2(r-1)(t-1)
\left(
\frac1r+\frac1t
\right).
\end{aligned}
\]
This is also a special case of the results of
Markowsky and Palacios \cite{MarkowskyPalacios2021}.

The number of edges is
\[
|E(G_{r,t})|
=
\frac{r(r-1)}{2}
+
\frac{t(t-1)}{2}.
\]
Hence,
\[
\HT(G_{r,t})
=
|E(G_{r,t})|\Kf(G_{r,t})
+
\mathcal A(G_{r,t}).
\]
Thus, the hitting time index can be obtained without calculating
all individual hitting times.
The resulting expression agrees with the corresponding formula in
\cite{PalaciosSantamaria2025}.

We next consider the case in which several identical complete graphs
share one common vertex.
Let $s\geq2$ and $r\geq2$, and let $G_{s,r}$ be obtained from
$s$ copies of $K_r$ by identifying one vertex from each copy.

\begin{proposition}
\label{prop:conjoined-identical-complete}

For $G_{s,r}$,
\[
\mathcal A(G_{s,r})
=
s(s-1)(r-1)^2.
\]
Moreover,
\[
\Kf(G_{s,r})
=
s(r-1)
+
\frac{2s(s-1)(r-1)^2}{r}.
\]
Consequently,
\[
\HT(G_{s,r})
=
\frac{s r(r-1)}{2}
\left\{
s(r-1)
+
\frac{2s(s-1)(r-1)^2}{r}
\right\}
+
s(s-1)(r-1)^2.
\]

\end{proposition}

\begin{proof}

Let $c$ be the common vertex.
Every noncentral vertex has degree $r-1$, while
\[
d(c)=s(r-1).
\]
By symmetry, all noncentral vertices have the same value of $R_d$.
For the common vertex,
\[
R_d(c)
=
s(r-1)^2\frac{2}{r}
=
\frac{2s(r-1)^2}{r}.
\]
Let $u$ be a noncentral vertex.
The contributions to $R_d(u)$ come from the common vertex,
the other vertices in the same copy of $K_r$,
and the vertices in the remaining $s-1$ copies.
Hence,
\[
\begin{aligned}
R_d(u)
&=
s(r-1)\frac{2}{r}
+
(r-2)(r-1)\frac{2}{r}
+
(s-1)(r-1)^2\frac{4}{r}.
\end{aligned}
\]
Therefore,
\[
R_d(u)-R_d(c)
=
2(s-1)(r-1).
\]
There are $s(r-1)$ noncentral vertices. Hence,
\[
\begin{aligned}
\mathcal A(G_{s,r})
&=
\frac12
s(r-1)\cdot2(s-1)(r-1)\\
&=
s(s-1)(r-1)^2.
\end{aligned}
\]

For the Kirchhoff index, the pairs contained in a single copy of $K_r$
contribute
\[
s\binom{r}{2}\frac{2}{r}
=
s(r-1),
\]
while pairs of noncentral vertices belonging to different copies contribute
\[
\binom{s}{2}(r-1)^2\frac{4}{r}
=
\frac{2s(s-1)(r-1)^2}{r}.
\]
Therefore,
\[
\Kf(G_{s,r})
=
s(r-1)
+
\frac{2s(s-1)(r-1)^2}{r}.
\]
Finally,
\[
|E(G_{s,r})|
=
s\frac{r(r-1)}{2},
\]
and the formula for $\HT(G_{s,r})$ follows from
Theorem~\ref{thm:HT-decomposition}.

\end{proof}

The formulas above give another derivation of the corresponding
hitting time indices considered in \cite{PalaciosSantamaria2025}.

\section{Conclusion}
\label{sec:conclusion}
In this paper, we considered the decomposition
\[
\HT(G)=m\Kf(G)+\mathcal A(G),
\]
where
\[
\mathcal A(G)
=
\frac12
\sum_{\{u,v\}\subseteq V}
|H_{uv}-H_{vu}|
=
\frac12
\sum_{\{u,v\}\subseteq V}
|R_d(u)-R_d(v)|.
\]
Thus, the decomposition provides a direct relation
between a random-walk-based graph invariant
and classical resistance- and distance-based graph invariants.
By ordering the values of $R_d$, we obtained
\[
\mathcal A(G)
=
\frac12
\sum_{k=1}^{n-1}k(n-k)\Delta_k,
\]
which yields lower and upper bounds for $\mathcal A(G)$
in terms of $R_d^{\max}-R_d^{\min}$,
together with their equality conditions.

For trees, we obtained the two-sided estimate
\[
(n-1)(n-2)
\leq
\mathcal A(T)
\leq
(n-2)W(T).
\]
The lower bound is attained if and only if
\[
T\cong K_{1,n-1}.
\]
Consequently,
\[
(n-1)W(T)+(n-1)(n-2)
\leq
\HT(T)
\leq
(2n-3)W(T).
\]
Together with the known minimum of the Wiener index among trees,
the lower bound also recovers that the star minimizes
the hitting time index among trees of fixed order.

We also calculated the asymmetry term for complete bipartite graphs
and paths, and applied the decomposition to graphs obtained
by conjoining complete graphs.
In the latter examples, known hitting-time formulas can be recovered
from effective-resistance calculations without determining
all individual hitting times.

Several questions remain for further study.
In particular, it would be interesting to determine a sharper
upper bound for $\mathcal A(T)$ among trees of fixed order.
The explicit formulas for paths suggest the question of whether
\[
\mathcal A(T)\leq\mathcal A(P_n)
\]
holds for every tree $T$ on $n$ vertices,
with equality only for the path $P_n$.
Other natural questions concern the behavior of $\mathcal A(G)$
under graph operations and its determination for graph families
in which the degree-weighted resistance sums take only a few
distinct values.

\section*{Statements and Declarations}

\subsection*{Funding}
The authors declare that no funds, grants, or other support were received
during the preparation of this manuscript.

\subsection*{Competing Interests}
The authors have no relevant financial or non-financial interests to disclose.

\subsection*{Data Availability}
Data sharing is not applicable to this article as no datasets were generated
or analysed during the current study.



\begin{thebibliography}{99}

\bibitem{ChangXu2017}

X.~Chang and H.~Xu,
Chung--Yau invariants and graphs with symmetric hitting times,
\textit{J. Graph Theory}
\textbf{85}, 691--705 (2017).
\url{https://doi.org/10.1002/jgt.22099}


\bibitem{ChandraEtAl1989}

A.~K.~Chandra, P.~Raghavan, W.~L.~Ruzzo,
R.~Smolensky and P.~Tiwari,
The electrical resistance of a graph captures its commute and cover times,
in: \textit{Proceedings of the Twenty-First Annual ACM Symposium on Theory of Computing},
pp.~574--586 (1989).
\url{https://doi.org/10.1145/73007.73062}


\bibitem{Rada2019}

R.~Cruz and J.~Rada,
The path and the star as extremal values of vertex-degree-based
topological indices among trees,
\textit{MATCH Commun. Math. Comput. Chem.}
\textbf{82}, 715--732 (2019).


\bibitem{DobryninEntringerGutman2001}

A.~A.~Dobrynin, R.~Entringer and I.~Gutman,
Wiener index of trees: theory and applications,
\textit{Acta Appl. Math.}
\textbf{66}, 211--249 (2001).
\url{https://doi.org/10.1023/A:1010767517079}


\bibitem{DoyleSnell1984}

P.~G.~Doyle and J.~L.~Snell,
\textit{Random Walks and Electrical Networks},
Carus Mathematical Monographs, Vol.~22,
Mathematical Association of America,
Washington, DC (1984).


\bibitem{GutmanMohar1996}

I.~Gutman and B.~Mohar,
The quasi-Wiener and the Kirchhoff indices coincide,
\textit{J. Chem. Inf. Comput. Sci.}
\textbf{36}, 982--985 (1996).
\url{https://doi.org/10.1021/ci960007t}


\bibitem{HeEtAl2026}

M.~He, Y.~Mao, B.~Furtula and X.~Zhang,
Hitting time index of trees,
\textit{J. Math. Chem.}
\textbf{64}, Article~7 (2026).
\url{https://doi.org/10.1007/s10910-025-01767-0}


\bibitem{KleinRandic1993}

D.~J.~Klein and M.~Randi\'c,
Resistance distance,
\textit{J. Math. Chem.}
\textbf{12}, 81--95 (1993).
\url{https://doi.org/10.1007/BF01164627}


\bibitem{MarkowskyPalacios2021}

G.~Markowsky and J.~L.~Palacios,
Kemeny's constant and Kirchhoffian indices for conjoined highly symmetric graphs,
\textit{Discrete Appl. Math.}
\textbf{302}, 215--220 (2021).
\url{https://doi.org/10.1016/j.dam.2021.07.007}


\bibitem{OrlicEtAl2026}

S.~Orli\'c, J.~L.~Palacios and A.~Petojevi\'c,
Hitting time index for broom graphs,
\textit{AppliedMath}
\textbf{6}, Article~93 (2026).
\url{https://doi.org/10.3390/appliedmath6060093}


\bibitem{Palacios2025}

J.~L.~Palacios,
A new probabilistic molecular index,
\textit{MATCH Commun. Math. Comput. Chem.}
\textbf{93}, 499--507 (2025).
\url{https://doi.org/10.46793/match.92-2.499P}


\bibitem{PalaciosPetojevic2026}

J.~L.~Palacios and A.~Petojevi\'c,
An edge operation that increases the hitting time index of trees,
\textit{Mathematics}
\textbf{14}, Article~2525 (2026).
\url{https://doi.org/10.3390/math14142525}


\bibitem{PalaciosSantamariaGraphs2025}

J.~L.~Palacios and A.~Santamar\'ia,
On the hitting time index of graphs,
\textit{Open J. Appl. Sci.}
\textbf{15}, 2464--2478 (2025).
\url{https://doi.org/10.4236/ojapps.2025.158164}


\bibitem{PalaciosSantamaria2025}

J.~L.~Palacios and A.~Santamar\'ia,
Some new results for the hitting time index,
\textit{MATCH Commun. Math. Comput. Chem.}
\textbf{94}, 427--445 (2025).
\url{https://doi.org/10.46793/match.94-2.427P}


\bibitem{PetojevicPalaciosOrlic2026}

A.~Petojevi\'c, J.~L.~Palacios and S.~Orli\'c,
A new formula for the hitting time index of broom graphs,
\textit{Mathematics}
\textbf{14}, Article~2174 (2026).
\url{https://doi.org/10.3390/math14122174}


\bibitem{Tetali1991}

P.~Tetali,
Random walks and the effective resistance of networks,
\textit{J. Theor. Probab.}
\textbf{4}, 101--109 (1991).
\url{https://doi.org/10.1007/BF01046996}

\end{thebibliography}
\end{document}